\documentclass[12pt]{amsart}
\usepackage{amssymb}
\usepackage{tikz-cd,adjustbox,dsfont}
\usepackage[all]{xy}

\usepackage[colorlinks=true,linkcolor=magenta,citecolor=magenta]{hyperref}

\newtheorem{theorem}{Theorem}[section]
\newtheorem{definition}[theorem]{Definition}
\newtheorem{proposition}[theorem]{Proposition}

\newtheorem{lemma}[theorem]{Lemma}
\begin{document}

\title[Orbitals of the quantum permutation group]{Tracing the orbitals of the quantum permutation group}

\author{J.P. McCarthy}
\address{%
Department of Mathematics\\
Munster Technological University\\
Cork\\
Ireland}

\email{jp.mccarthy@mtu.ie}
\date{March 4, 2023}
\begin{abstract}
Using a suitably noncommutative flat matrix model, it is shown that the quantum permutation group has free orbitals: that is, a monomial in the generators of the algebra of functions can  be zero for trivial reasons only. It is shown that any strictly intermediate quantum subgroup between the classical and quantum permutation groups must have free three-orbitals. This is used to give explicit formulae for the Haar state on degree four monomials that hold for  such intermediate quantum subgroups as well as the quantum permutation group itself.
\end{abstract}

\maketitle

\section{Introduction \& Preliminaries}
In 1995 Alain Connes asked the question ``What is the quantum automorphism group of a space?'' and in 1998 Wang \cite{wa2} answered the question in the case of finite quantum spaces, i.e. the abstract spectra of finite dimensional $\mathrm{C}^\ast$-algebras. This included notably the quantum automorphism groups of finite classical spaces, which can be viewed equivalently as the quantum automorphism group of $\{1,\dots,N\}$ (that preserves the uniform measure), or as the quantum permutation group on $N$ points. This is all in the language of compact matrix quantum groups \cite{wo1}:
\begin{definition}(Woronowicz)
  If a unital $\mathrm{C}^\ast$-algebra $C(\mathbb{G})$ is:
  \begin{enumerate}
    \item[(i)] generated by the entries of a unitary matrix $u\in M_N(C(\mathbb{G}))$, and
    \item[(ii)] $u$ and $u^t$ are invertible, and
    \item[(iii)] there exists a $*$-homomorphism $\Delta:C(\mathbb{G})\to C(\mathbb{G})\underset{\min}{\otimes}C(\mathbb{G})$, such that
    $$\Delta(u_{ij})= \sum_{k=1}^Nu_{ik}\otimes u_{kj},$$
  \end{enumerate}
  then $\mathbb{G}$ is a compact matrix quantum group with fundamental representation $u\in M_N(C(\mathbb{G}))$.
\end{definition}

 Conventionally, compact matrix quantum groups with noncommutative algebras of functions are spoken about only via their algebra of continuous functions: the quantum group is a so-called \emph{virtual object}.  The algebra of continuous functions $C(S_N^+)$ on the quantum permutation group $S_N^+$ is the universal $\mathrm{C}^\ast$-algebra generated by the entries of an $N\times N$ magic unitary $u\in M_N(C(S_N^+))$, that is a matrix whose rows and columns are partitions of unity, that is they consist of projections, $u_{ij}=u_{ij}^\ast=u_{ij}^2$, that sum to the identity on rows and columns.  Thus, with the unit of $C(S_N^+)$ denoted by $\mathds{1}_{S_N^+}:=1_{C(S_N^+)}$,
$$\sum_{k=1}^Nu_{ik}=\mathds{1}_{S_N^+}=\sum_{k=1}^Nu_{kj}.$$
It can be shown that $S_N^+=S_N$ for $N\leq 3$; however for $N\geq 4$, the quantum permutation group $S_N^+$ is non-classical and infinite in the sense that $C(S_N^+)$ is noncommutative and infinite dimensional \cite{ba1}.

\bigskip

If $\mathbb{G}$ is a compact matrix quantum group with magic fundamental representation $v\in M_N(C(\mathbb{G}))$, the universal property of $C(S_N^+)$ gives $\pi:C(S_N^+)\to C(\mathbb{G})$ a surjective $*$-homomorphism, $u_{ij}\mapsto v_{ij}$, that respects the comultiplication:
$$\Delta_{C(\mathbb{G})}\circ \pi=(\pi\otimes \pi)\circ \Delta_{C(S_N^+)}.$$
That is to say that $\mathbb{G}\subseteq S_N^+$ is a quantum subgroup.  The classical permutation group $S_N$ is a compact matrix quantum group and, where $\mathds{1}_{j\to i}(\sigma):=\delta_{i,\sigma(j)}$, it is a quantum subgroup $S_N\subseteq S_N^+$ via the magic fundamental representation:
$$v=(\mathds{1}_{j\to i})_{i,j=1}^N.$$
 Banica and Bichon \cite{bb3}, through classifying the quantum subgroups $\mathbb{G}\subseteq S_4^+$, noted that $S_4\subset S_4^+$ is a maximal quantum subgroup, and conjectured that $S_N\subseteq S_N^+$ is a  maximal quantum subgroup at all $N$. Only recently did Banica \cite{ban} use advances in subfactor theory to show that $S_5\subset S_5^+$ is also a maximal quantum subgroup. As $S_N^+=S_N$ for $N\leq 3$, the current state of art is:
\begin{theorem}
For $N\leq 5$, the classical permutation group $S_N$ is a maximal quantum subgroup of the quantum permutation group $S_N^+$.
\end{theorem}
This work humbly posits the existence of an \emph{exotic} intermediate quantum permutation group:
$$S_N\subsetneq \mathbb{G}\subsetneq S_N^+,$$
and studies some of its very basic  algebraic properties. Using the abelianisation $\pi_{\text{ab}}:C(\mathbb{G})\to C(S_N)$, each $\sigma\in S_N$ gives a character on (universal) $C(\mathbb{G})$:
$$\operatorname{ev}_{\sigma}(f)=\pi_{\text{ab}}(f)(\sigma).$$
These characters can be used to permute labels in the following sense:
\begin{proposition}\label{permute}
Suppose $S_N\subseteq \mathbb{G}\subseteq S_N^+$. Then for all states $\varphi$ on $C(\mathbb{G})$ and $\sigma,\tau\in S_N$:
$$(\operatorname{ev}_{\sigma^{-1}}\star \varphi\star \operatorname{ev}_\tau)(u_{i_1j_1}\cdots u_{i_nj_n})=\varphi(u_{\sigma(i_1)\tau(j_1)}\cdots u_{\sigma(i_n)\tau(j_n)}).$$
\end{proposition}
\begin{proof}
This is a slight generalisation of (Prop. 6.4, \cite{mcc}), albeit with the same proof.
\end{proof}

\section{Free orbitals}
The orbitals of a quantum permutation group $\mathbb{G}\subseteq S_N^+$ are related to non-zero monomials:
$$u_{i_1j_1}\cdots u_{i_mj_m}\neq 0.$$
The spectre of one-orbitals, or orbits, is present in the work of Bichon (Prop. 4.1,\cite{bi3}). However the first explicit mention of $u_{ij}\neq 0$ relating to orbits is in the PhD thesis of Huang  \cite{hu0}.  However it was Lupini, Man\v{c}inska, \& Roberson \cite{lmr} who first defined two-orbitals:
\begin{definition}
Let $\mathbb{G}\subseteq S_N^+$. Define a relation $\sim_m$ on $\{1,2,\dots,N\}^m$ by
$$(i_1,\dots,i_m)\sim_m(j_1,\dots,j_m)\iff u_{i_1j_1}\cdots u_{i_mj_m}\neq 0.$$
\end{definition}
The $m$-orbital relation is reflexive and symmetric. Both $\sim_1$ and $\sim_2$ are equivalence relations, their equivalence classes called orbits and two-orbitals \cite{lmr}. Higher orbitals are not in general equivalence relations \cite{mcc}.

\bigskip

Let $u\in M_N(\mathcal{A})$ be a magic unitary with entries in a $\mathrm{C}^*$-algebra. As the rows and columns of $u$ are partitions of unity, entries along rows or columns are orthogonal:
$$u_{ik}u_{il}=\delta_{k,l}\,u_{ik}\text{ and }u_{kj}u_{lj}=\delta_{k,l}\,u_{kj}.$$
If $u_{i_1j_1}\cdots u_{i_mj_m}=0$ because two consecutive terms are along a common row or column, say that  the monomial $u_{i_1j_1}\cdots u_{i_mj_m}$ is zero \emph{for trivial reasons}.
\begin{definition}
A quantum permutation group $\mathbb{G}\subseteq S_N^+$ has \emph{free $m$-orbitals} if
$$u_{i_1j_1}\cdots u_{i_mj_m}=0$$
for trivial reasons only. A quantum permutation group has \emph{free} orbitals if it has free $m$-orbitals for all $m\geq 1$.
\end{definition}
$S_N\subsetneq S_N^+$ has free orbits and free two-orbitals but it does not have free three-orbitals, e.g.
$$\mathds{1}_{1\to 3}\mathds{1}_{2\to 2}\mathds{1}_{1\to 1}=0.$$

By the universal property of $C(S_N^+)$, there is a $*$-homomorphism from $C(S_N^+)$ to any $\mathrm{C}^*$-algebra $\mathcal{A}$ that admits an $N\times N$ magic unitary. To show that $S_N^+$ has free orbitals, it suffices to produce a magic unitary such that monomials with entries from $v\in M_N(\mathcal{A})$ are zero for trivial reasons only. This is because, where $\pi:C(S_N^+)\to \mathcal{A}$ is given by the universal property,
\begin{equation}\pi(u_{i_1j_1}u_{i_2j_2}\cdots u_{i_mj_m})=v_{i_1j_1}v_{i_2j_2}\cdots v_{i_mj_m}.\label{pio}\end{equation}

 Studied by Banica and Nechita \cite{bni}, a flat matrix model $v\in M_N(M_N(\mathbb{C}))$ is a magic unitary of rank one projections. Such a model can be given by a \emph{magic basis}, a matrix $\xi\in M_N(\mathbb{C}^N)$ such that each row and column of $\xi$ forms an orthonormal basis of $\mathbb{C}^N$:
$$v_{ij}=|\xi_{ij}\rangle \langle \xi_{ij}|.$$
Monomials from such magic unitaries are very easy to handle:
$$v_{i_1j_1}v_{i_2j_2}\cdots v_{i_mj_m}=|\xi_{i_1j_1}\rangle \langle \xi_{i_1j_1}|\xi_{i_2j_2}\rangle \langle \xi_{i_2j_2}|\cdots |\xi_{i_mj_m}\rangle \langle \xi_{i_mj_m}|,$$
and equal to zero only when some $\langle \xi_{i_nj_n}|\xi_{i_{n+1}j_{n+1}}\rangle=0$. Also, the commutativity of the projections is determined by inner products:
$$[v_{ij},v_{kl}]\neq 0\iff 0<|\langle \xi_{ij}|\xi_{kl} \rangle|<1,$$
and putting these facts together, if there exists a suitably noncommutative  flat matrix model $v\in M_N(M_N(\mathbb{C}))$ of $C(S_N^+)$, one in which:
$$[v_{ij},v_{kl}]=0\iff i=k\text{ or }j=l,$$
then $v$ has free orbitals in the sense that:
$$v_{i_1j_1}v_{i_2j_2}\cdots v_{i_mj_m}=0$$
for trivial reasons only, and thus by (\ref{pio}) so does $S_N^+$.
\begin{theorem}\label{4fo}
$S_N^+$ has free orbitals for $N\geq 4$.
\end{theorem}
\begin{proof}
After a change of basis, the
$$\begin{pmatrix}
     \frac{1}{\sqrt{3}}e^{\pi i/4}  & -\sqrt{\frac{2}{3}}e^{-\pi i/2} \\
      \sqrt{\frac{2}{3}}e^{\pi i/2} & \frac{1}{\sqrt{3}}e^{-\pi i/4}
    \end{pmatrix}\in SU(2)$$ fibre of the Pauli representation  $C(S_4^+)\to C(SU(2),M_4(\mathbb{C}))$ \cite{bac} yields a suitably noncommutative flat matrix model for $C(S_4^+)$, given by the magic basis:
    $$\xi=\frac{1}{3}\begin{bmatrix}
                       3e_1 & e_2-2e_3-2e_4 & e_4-2e_2-2e_3 & e_3-2e_2-2e_4 \\
                       3e_2 & e_1-2e_3+2e_4 & e_3-2e_1+2e_4 & e_4+2e_1+2e_3 \\
                       3e_3 & e_4-2e_1+2e_2 & e_2+2e_1+2e_4 & e_1+2e_2-2e_4 \\
                       3e_4 & e_3+2e_1+2e_2 & e_1-2e_2+2e_3 & e_2-2e_1+2e_3
                     \end{bmatrix}.$$
Now for each $N\geq 5$, let $\omega=\exp(2\pi i/N)$. Where $e_1,e_2,\dots, e_N$ are the standard basis vectors of $\mathbb{C}^N$, for $1\leq i,j\leq N$, define a vector $\xi_{ij}\in \mathbb{C}^N$ according to:
$$
\langle e_p,\xi_{ij}\rangle=\begin{cases}
  \dfrac{1}{\sqrt{N}}\omega^{1-j}, & \mbox{if } p=1, \\[2ex]
  \dfrac{1}{\sqrt{N}}\omega^{i-1}, & \mbox{if } p=N, \\[2ex]
  \dfrac{1}{\sqrt{N}}\omega^{p(i-j)}, & \mbox{otherwise}.
\end{cases}
$$
Note that $\xi_{ij}$ is a unit vector, and
\begin{align*}
\langle \xi_{ij},\xi_{kl}\rangle&=\frac{1}{N}\omega^{j-1}\cdot \omega^{1-l}+\frac{1}{N}\omega^{1-i}\cdot \omega^{k-1}+\frac{1}{N}\sum_{p=2}^{N-1}\omega^{-p(i-j)}\cdot \omega^{p(k-l)}
\\&=\frac{1}{N}\omega^{j-l}+\frac{1}{N}\omega^{k-i}-\frac{1}{N}-\frac{1}{N}\omega^{(k-i)+(j-l)}+\frac{1}{N}\sum_{p=0}^{N-1}\omega^{p((k-i)+(j-l))}
\\&=\frac{1}{N}(\omega^{j-l}-1)(1-\omega^{k-i})+\frac{1}{N}\sum_{p=0}^{N-1}\left[\omega^{(k-i)+(j-l)}\right]^p.
\\&=\begin{cases}
                                      1, & \mbox{if } i=k,j=l \\
                                      0, & \mbox{if }i=k,j\neq l\text{, or }i\neq k,j=l.
                                    \end{cases}\end{align*}
Therefore $\xi$ is a magic basis. Suppose now that $i\neq k$, $j\neq l$. Consider two cases:
\newline \textbf{Case 1:} If $(k-i)+(j-l)\equiv 0\mod N$ the sum is 1:
$$\langle \xi_{ij},\xi_{kl}\rangle=1+\frac{1}{N}(\omega^{j-l}-1)(1-\omega^{k-i}).$$
However $j-l\equiv i-k \mod N$ giving:
$$\langle \xi_{ij},\xi_{kl}\rangle=1+\frac{1}{N}(2\Re(\omega^{k-i})-2)\implies 1-\frac{4}{N}\leq \langle \xi_{ij},\xi_{kl}\rangle<1,$$
as $i\neq k$. Note that because $N>4$, this is non-zero.
\newline \textbf{Case 2:} If $(k-i)+(j-1)\not\equiv 0\mod N$ the sum is zero:
\begin{align*}
\langle \xi_{ij},\xi_{kl}\rangle&=\frac{1}{N}(\omega^{j-l}-1)(1-\omega^{k-i})
\\ \implies |\langle \xi_{ij},\xi_{kl}\rangle|&=\frac{1}{N}|\omega^{j-l}-1||1-\omega^{k-i}|\leq \frac{4}{N}<1.
\end{align*}
Neither is $\langle \xi_{ij},\xi_{kl}\rangle=0$ because $i\neq k$, $l\neq j$. Therefore $\xi$ is a suitably noncommutative magic basis, exhibiting free orbitals for $C(S_N^+)$ in the range $N\geq 5$.
\end{proof}
For a fixed $N\geq 4$, the magic basis given above gives a matrix model $\pi:C(S_N^+)\to M_N(\mathbb{C})$. Is it \emph{inner faithful}? There is a largest quantum permutation group, the Hopf image \cite{bb4}, $\mathbb{G}_{\pi}\subseteq S_N^+$ which factorises through the quotient:
$$C(S_N^+)\to C(\mathbb{G}_{\pi})\to M_N(\mathbb{C}),$$
say $\pi=\rho\circ q$. If $q$ is an isomorphism, then $\mathbb{G}_{\pi}=S_N^+$, and the model is said to be \emph{inner faithful}. Where $\operatorname{tr}$ is the normalised trace on $M_N(\mathbb{C})$, Wang \cite{wa1} showed that:
$$h_{C(\mathbb{G}_{\pi})}=w^*\text{-}\lim_{n\to \infty}\frac{1}{n}\sum_{k=1}^n (\operatorname{tr}\circ q)^{\star k}.$$
The inner faithfulness of a Theorem \ref{4fo} model would complement the Connes embedding related results of \cite{BCF}. Any attempt to show that the matrix model in Theorem \ref{4fo} is inner faithful might use the following:
\begin{theorem}\label{PWT}
Where $u\in M_N(C(S_N^+))$ is the fundamental representation of $S_N^+$, the moments of the main character $\operatorname{fix}=\sum_i u_{ii}$ are given by the Catalan numbers:
$$h(\operatorname{fix}^k)=C_k.$$
Furthermore, if $\mathbb{G}\subseteq S_N^+$ via $v\in M_N(C(\mathbb{G}))$, with main character $\operatorname{fix}_{\mathbb{G}}=\sum_i v_{ii}$, then the moments are whole numbers, and
$$h_{C(\mathbb{G})}(\operatorname{fix}_{\mathbb{G}}^k)=C_k\implies \mathbb{G}=S_N^+.$$
\end{theorem}
This is a well known result, whose proof relies upon Peter--Weyl theory (see, for example, (Th. 9.10, Th. 3.21, Prop. 4.3, \cite{ba1}). The notation $\operatorname{fix}$ is used because the classical analogue $\sum_{i}\mathds{1}_{i\to i}$ counts the number of fixed points of a permutation.

\section{The orbitals of exotic quantum permutation groups}
In this section it is shown that if an exotic $S_N\subsetneq \mathbb{G}\subseteq S_N^+$ exists, then it has free three-orbitals.

\begin{proposition}\label{orbit}
  Exotic quantum permutation groups $S_N\subsetneq \mathbb{G}\subsetneq S_N^+$ have free orbits and free two-orbitals.
\end{proposition}
\begin{proof}
These are inherited from $S_N$ via the abelianisation.
\end{proof}
\begin{proposition}\label{orbital}
Consider exotic $S_N\subsetneq \mathbb{G}\subsetneq S_N^+$. Entries from $u\in M_N(C(\mathbb{G}))$ pairwise-commute only when they are from the same row or column.
\end{proposition}
\begin{proof}
By assumption $C(\mathbb{G})$ is noncommutative and therefore there exists a non-commuting pair (necessarily $a\neq c$, $b\neq d$):
$$u_{ab}u_{cd}\neq u_{cd}u_{ab}.$$
If $p,q$ are projections in a $\mathrm{C}^*$-algebra then $pq=qp$ if and only if $|pq|^2=|qp|^2$, therefore:

$$|u_{ab}u_{cd}|^2\neq |u_{cd}u_{ab}|^2.$$
The states on a $\mathrm{C}^*$-algebra are separating, and therefore there exists $\varphi_0$ on $C(\mathbb{G})$ such that:
$$\varphi_0(|u_{cd}u_{ab}|^2)\neq \varphi_0(|u_{ab}u_{cd}|^2).$$
Let $\sigma,\,\tau\in S_N$ be such that:
$$\sigma(i)=c,\,\sigma(k)=a\text{ and }\tau(j)=d,\,\tau(l)=b.$$
With  $\varphi:=\operatorname{ev}_{\sigma^{-1}}\star \varphi_0 \star \operatorname{ev}_{\tau},$ using Proposition \ref{permute}:
$$\varphi(|u_{ij}u_{kl}|^2)=\varphi_0(|u_{cd}u_{ab}|^2)\neq \varphi_0(|u_{ab}u_{cd}|^2) =\varphi(|u_{kj}u_{ij}|^2).$$
Therefore
$$|u_{ij}u_{kl}|^2\neq |u_{kl}u_{ij}|^2\implies  u_{ij}u_{kl}\neq u_{kl}u_{ij}\qquad\qedhere$$
\end{proof}

\begin{theorem}\label{efto}
 Exotic quantum permutation groups $S_N\subsetneq \mathbb{G}\subsetneq S_N^+$ have free three-orbitals.
\end{theorem}
\begin{proof}
Let $1\leq a,b,c,d,e,f\leq N$. Where $u\in M_N(C(\mathbb{G}))$, it is required to show
$$\delta_{ac}+\delta_{bd},\delta_{ce}+\delta_{df}\in\{0,2\}\implies u_{ab}u_{cd}u_{ef}\neq 0.$$
To capture all the possibilities, consider:
$$(r,s,t)=(\delta_{ac}+\delta_{bd},\delta_{ce}+\delta_{df},\delta_{ae}+\delta_{bf}).$$
The only non-trivial case is $(r,s,t)=(0,0,1)$. In this case, where different `$i$' and `$j$' symbols are distinct, it is $u_{ab}u_{cd}u_{af}$ or $u_{ab}u_{cd}u_{eb}$. The first case can be reduced to the second with the use of the antipode. So consider $u_{ab}u_{cd}u_{eb}\in C(\mathbb{G})$. With $(r,s)=(0,0)$, $u_{cd}u_{eb}\neq 0$. By Proposition \ref{orbital}
$$u_{cd}u_{eb}\neq u_{eb} u_{cd}.$$
 Represent $C(\mathbb{G})$ using the universal GNS representation $\pi_{\text{GNS}}(C(\mathbb{G}))\subset B(\mathsf{H})$. Denote
$$p:=\pi_{\text{GNS}}(u_{eb})\text{ and }q:=\pi_{\text{GNS}}(u_{cd}).$$
As $pq\neq qp$,  there exists $x\in \operatorname{ran}p$ orthogonal to both\footnote{in the notation of (\cite{bos},(1)), $x\in M_0$} $\operatorname{ran}p\cap\operatorname{ran}q$ and $\operatorname{ran} p\cap \ker q$. Define a state on $C(\mathbb{G})$:
$$\varphi_0(f)=\langle x,\pi_{\text{GNS}}(f)x\rangle.$$
Consider:
\begin{align}
  \varphi_0(u_{eb}fu_{eb})  &=\langle x,\pi_{\text{GNS}}(u_{eb}fu_{eb})x\rangle
    =\langle x,p\pi_{\text{GNS}}(f)px\rangle\nonumber\\
  &=\langle px,\pi_{\text{GNS}}(f)px\rangle
    =\langle x,\pi_{\text{GNS}}(f)x\rangle=\varphi_0(f)\label{note},
\end{align}
as $x\in\operatorname{ran}p$. Furthermore, together with $x\in\operatorname{ran}p$
\begin{align*}
 \varphi_0(u_{cd})= \langle x,qx\rangle =1 & \implies x\in\operatorname{ran}q \\
  \varphi_0(u_{cd})=\langle x,qx\rangle =0 & \implies x\in\ker q
\end{align*}
but $x$ is orthogonal to both $\operatorname{ran}p\cap\operatorname{ran}q$ and $\operatorname{ran} q\cap \ker q$ thus
$$0<\langle x,qx\rangle<1\implies 0<\varphi_0(u_{cd})<1.$$
Now define a state
$$\varphi(f):=\frac{\varphi_0(u_{cd}f u_{cd})}{\varphi_0(u_{cd})}=\frac{\langle qx,\pi_{\text{GNS}}(f)qx \rangle}{\langle qx,qx\rangle}.$$
In particular
\begin{align*}
  \varphi(u_{eb}) & =\frac{\langle qx,pqx\rangle}{\langle qx,qx\rangle}
\end{align*}
Together with $qx\in\operatorname{ran}q$:
\begin{align*}
  \varphi(u_{eb})=1 & \implies qx\in\operatorname{ran}p \\
  \varphi(u_{eb})=0 & \implies qx\in\operatorname{ker}p
\end{align*}
But $qx$ is orthogonal to $\operatorname{ran}p\cap \operatorname{ran}q$, and $\operatorname{ker}p\cap \operatorname{ran} q$ and it follows that:
$$0<\varphi(u_{eb})<1.$$
However
$$\varphi(u_{eb})+\varphi\left(\sum_{i\neq e}u_{ib}\right)=\varphi\left(\sum_iu_{ib}\right)=\varphi(\mathds{1}_{S_N^+})=1,$$
therefore there exists $u_{ab}\neq u_{eb}$ such that:
\begin{align*}
\varphi(u_{ab})&>0
\\ \implies \frac{\varphi_0(u_{cd}u_{ab}u_{cd})}{\varphi_0(u_{cd})}&>0
\\ \underset{(\ref{note})}{\implies} \frac{\varphi_0(u_{eb}u_{cd}u_{ab}u_{cd}u_{eb})}{\varphi_0(u_{cd})}&>0
\\ \implies \varphi_0(|u_{ab}u_{cd}u_{eb}|^2)&>0
\\ \implies u_{ab}u_{cd}u_{eb}&\neq 0\qedhere
\end{align*}
\end{proof}
This is the published version of the proof. It also follows from the following elementary lemma:
\begin{lemma}
  Let $\{p_k\}_{k=1}^n\subset \mathcal{A}$ be a partition of unity in a $\mathrm{C}^*$-algebra, and $q\in \mathcal{A}$ a projection such that $[p_i,q]\neq 0$. Then there exists $p_j\neq p_i$ such that $p_jqp_j\neq 0$.
\end{lemma}
\begin{proof}
Assume that for all $j\neq i$ the $p_jqp_i$ are zero. Add them up to get $(1-p_i)qp_i=0$. This yields $qp_i=p_iqp_i$ which implies, taking adjoints, that $p_i$ commutes with $q$. Therefore one of the $p_jqp_i$ must be zero.
\end{proof}
Is there a quantum permutation group $\mathbb{G}\subsetneq S_N^+$ with free three-orbitals? Note that if a Theorem \ref{4fo} model $\pi:C(S_N^+)\to M_N(\mathbb{C})$ is \emph{not} inner faithful, then the Hopf image $\mathbb{G}_\pi\subsetneq S_N^+$ has free orbitals, and in particular free three-orbitals.
\section{The Haar state}
The value of the Haar state at monomials is ostensibly  important in the theory of quantum permutation groups. The values of the Haar state on degree three monomials in $C(S_N^+)$,
$$h(u_{i_1j_1}u_{i_2j_2}u_{i_3j_3}),$$
are well known, but their calculation typically uses representation theory, usually via a study of the fixed point spaces of tensor powers of the fundamental representation  \cite{ba2}. However, using Proposition \ref{permute}, these can be calculated using elementary considerations. Furthermore, while nothing is known about their representation theory, these calculations also hold for exotic quantum permutation groups. The same elementary considerations can be used to derive relations between degree four monomials in the exotic case:
$$h(u_{i_1j_1}u_{i_2j_2}u_{i_3j_3}u_{i_4j_4}).$$
Representation theory, via Theorem \ref{PWT}, can be used to give another relation from which explicit formulae follow.
\begin{proposition}\label{Haar}
The Haar state on $C(S_N^+)$, and exotic $C(\mathbb{G})$, is tracial, invariant under the antipode:
    \begin{equation}h(u_{i_1j_1}\cdots u_{i_nj_n})=h(u_{j_ni_n}\cdots u_{j_1i_1}),\label{antip}\end{equation}
    and invariant under permutations of the labels, for $\sigma,\tau\in S_N$:
    \begin{equation}h(u_{i_1j_1}\cdots u_{i_nj_n})=h(u_{\sigma(i_1),\tau(j_1)}\cdots u_{\sigma(i_n),\tau(j_n)}).\label{relab}\end{equation}
\end{proposition}
\begin{proof}
That the Haar state is tracial, and invariant under the antipode is standard. Apply Proposition \ref{permute} with $h=\operatorname{ev}_{\sigma^{-1}}\star h\star \operatorname{ev}_{\tau}$ for invariance under permutations of the labels.
\end{proof}
Where $u$ is a magic unitary,  a monomial $f=u_{i_1j_1}\cdots u_{i_mj_m}$  is in reduced form if it is zero, or if for all $1\leq n\leq m-1$:
$$\delta_{i_n,i_{n+1}}+\delta_{j_n,j_{n+1}}=2,$$
that is  use the relation $u_{ij}^2=u_{ij}$ and the orthogonality along rows and columns of $u$ to ensure that $f$ is of minimal degree. In the below, all monomials are assumed reduced.

\begin{proposition}
For both $S_N^+$ and exotic $\mathbb{G}\subsetneq S_N^+$
\begin{align*}
h(u_{ij})&=\frac{1}{N},
\\ h(u_{ij}u_{kl})&=\frac{1}{N(N-1)},
\end{align*}
and, if $|\{i_1,i_2,i_3\}|=|\{j_1,j_2,j_3\}|=3$:
  $$h(u_{i_1j_1}u_{i_2j_2}u_{i_3j_3})=\frac{1}{N(N-1)(N-2)}.$$

\end{proposition}
\begin{proof}
By Proposition \ref{Haar},  for any $1\leq j,k\leq N$, $h(u_{ij})=h(u_{ik})$. Therefore
\begin{align*}
h(\mathds{1}_{\mathbb{G}})=h\left(\sum_{k=1}^Nu_{ik}\right)=\sum_{k=1}^Nh( u_{ik})=1\implies  Nh(u_{ij})=1.
\end{align*}
For the second equation,
\begin{align*}
  \frac{1}{N}  =h(u_{ij})=h\left(u_{ij}\left(\sum_{p=1}^Nu_{pl}\right)\right) =\sum_{p\neq i}h(u_{ij}u_{pl})=(N-1)h(u_{ij}u_{kl}).
\end{align*}
 Note that, by traciality $h(u_{i_1j_1}u_{i_2j_2}u_{i_1j_1})=h(u_{i_1j_1}u_{i_2j_2})$, and, if $j_3\neq j_1$, $h(u_{i_1j_1}u_{i_2j_2}u_{i_1j_3})=0$. Consider
\begin{align*}
  \frac{1}{N(N-1)} & =h(u_{i_1j_1}u_{i_2j_2})=h\left(u_{i_1j_1}u_{i_2j_2}\sum_{p=1}^Nu_{i_3p}\right) \\
   & =\sum_{p\neq i_1,i_2}h(u_{i_1j_1}u_{i_2j_2}u_{i_3p})=\sum_{p\neq i_1,i_2}h(u_{i_1j_1}u_{i_2j_2}u_{i_3j_3})
   \\ \implies h(u_{i_1j_1}u_{i_2j_2}u_{i_3j_3})&=\frac{1}{N(N-1)(N-2)}\qedhere
\end{align*}

\end{proof}
Proposition \ref{Haar} sets out some invariances of the Haar state.  For example, using (\ref{relab}) with $\sigma=(325)(47)$ and $\tau=(2184)(36)$:
$$h(u_{12}u_{34}u_{56}u_{78})=h(u_{11}u_{22}u_{33}u_{44}).$$
By using invariance under the antipode (\ref{antip}), and traciality:
\begin{align*}
  h(u_{11}u_{22}u_{11}u_{23}) & = h(u_{32}u_{11}u_{22}u_{11}) =h(u_{11}u_{22}u_{11}u_{32}).
\end{align*}
By using traciality, and (\ref{relab}) with $\sigma=(12)$ and $\tau=(132)$:
$$h(u_{11}u_{22}u_{13}u_{22})=h(u_{22}u_{13}u_{22}u_{11})=h(u_{11}u_{22}u_{11}u_{23}).$$
Using brute force and these various invariances, excluding cases that reduce via traciality (e.g. $h(u_{11}u_{22}u_{33}u_{11})=h(u_{11}u_{22}u_{33})$), it can be seen there are only seven basic integrals of degree four monomials (the \emph{integral} of $f$ here   refers to $h(f)$).

\begin{theorem}
Where $q(N)=N(N-1)(N^2-3N+1)$, the following formulae for the Haar states of both exotic $\mathbb{G}$ and $S_N^+$ hold:
\begin{align*}
  h(u_{11}u_{22}u_{11}u_{22})&=\frac{2N-5}{q(N)} \\
 h(u_{11}u_{22}u_{11}u_{23})&=\frac{N-3}{q(N)} \\
  h(u_{11}u_{22}u_{11}u_{33})&=\frac{N-2}{q(N)}\\
 h(u_{11}u_{22}u_{13}u_{24})&=-\frac{1}{q(N)} \\
h(u_{11}u_{22}u_{13}u_{32})&=-\frac{N-3}{(N-2)q(N)} \\
h(u_{11}u_{22}u_{13}u_{34})&=\frac{1}{(N-2)q(N)} \\
h(u_{11}u_{22}u_{33}u_{44})&=\frac{N}{(N-2)q(N)}
\end{align*}
\end{theorem}
\begin{proof} The Haar state is faithful on the $*$-algebra generated by these generators, and by Theorem \ref{4fo}, $u_{22}u_{11}u_{23}\neq 0$:
\begin{align*}
  \alpha_1 &=h( u_{11}u_{22}u_{11}u_{22})=h( |u_{22}u_{11}u_{22}|^2), \\
  \alpha_2 & =h( u_{11}u_{22}u_{11}u_{23})=h( |u_{22}u_{11}u_{23}|^2), \\
  \alpha_3 & =h( u_{11}u_{22}u_{11}u_{33})=h( |u_{22}u_{11}u_{33}|^2), \\
  \alpha_4 & =h( u_{11}u_{22}u_{13}u_{24}), \\
  \alpha_5 & =h( u_{11}u_{22}u_{13}u_{32}), \\
  \alpha_6 & =h( u_{11}u_{22}u_{13}u_{34}), \\
  \alpha_7 & =h( u_{11}u_{22}u_{33}u_{44}).
\end{align*}
Liberally using Proposition \ref{Haar}, the following six expansions give a rank six linear system in $\alpha_1,\dots,\alpha_7$:
$$h\left(u_{11}u_{22}u_{11}\sum_k u_{2k}\right),\,h\left(u_{11}u_{22}u_{11}\sum_k u_{3k}\right),\,h\left(u_{11}u_{22}u_{13}\sum_k u_{2k}\right),$$
$$h\left(u_{11}u_{22}u_{13}\sum_k u_{3k}\right),\,h\left(u_{11}u_{22}u_{33}\sum_k u_{4k}\right),\,h\left(u_{11}u_{22}u_{33}\sum_k u_{2k}\right).$$
For example,
\begin{align*}
  0  =h\left(u_{11}u_{22}u_{13}\sum_{k=1}^N u_{2k}\right)&=0+h( u_{11}u_{22}u_{13}u_{22})+0+(N-3)\alpha_4 \\
  \implies \alpha_2+(N-3)\alpha_4  &=0.
\end{align*}

Note that
\begin{align*}
\operatorname{fix}^2&=\operatorname{fix}+\sum_{\underset{k\neq l}{k,l=1}}^N u_{kk}u_{ll}
\\ \implies\operatorname{fix}^3&=\operatorname{fix}^2+\operatorname{fix}\cdot \sum_{k\neq l}u_{kk}u_{ll}.
\end{align*}
Then from $C_3=5$ and $C_2=2$ (see Theorem \ref{PWT})  it follows $h(\operatorname{fix}\cdot \sum_{k\neq l}u_{kk}u_{ll})=3$. Finally
\begin{align*}
  \operatorname{fix}^4 & =\operatorname{fix}^3+\operatorname{fix}^2\cdot \sum_{k\neq l}u_{kk}u_{ll} \\
   & = \operatorname{fix}^3+(\operatorname{fix}+\sum_{i\neq j}u_{ii}u_{jj})\cdot \sum_{k\neq l}u_{kk}u_{ll}\\
   & =\operatorname{fix}^3+\operatorname{fix}\cdot \sum_{k\neq l}u_{kk}u_{ll}+\sum_{i\neq j}u_{ii}u_{jj}\cdot \sum_{k\neq l}u_{kk}u_{ll}.
\end{align*}
In the case of $C(S_N^+)$,  with $C_4=14$, applying the Haar state on both sides yields:
$$h\left(\sum_{i\neq j}u_{ii}u_{jj}\cdot \sum_{k\neq l}u_{kk}u_{ll}\right)=6.$$
In the double sum, there are $N(N-1)$ terms with integral equal to $\alpha_1$, $2N(N-1)(N-2)$ terms with integral equal to $\alpha_3$, $N(N-1)$ terms with integral equal to $h(u_{11}u_{22}u_{11})$, $N(N-1)(N-2)$ with integral equal to $h(u_{11}u_{22}u_{33})$, $N(N-1)(N-2)$ terms with integral equal to $h(u_{11}u_{22}u_{33}u_{11})$, and finally $N(N-1)(N-2)(N-3)$ terms with integral equal to $\alpha_7$. The explicit formulae follow.

\bigskip

The fourth moment of the main character in the exotic case could be 15 or 14, but in the case of the fourth moment being 15, the relation generated gives
$$h(u_{11}u_{22}u_{11}u_{23})=h(|u_{22}u_{11}u_{23}|^2)=0,$$
but this is impossible due to Theorem \ref{efto}. Therefore the fourth moment must be 14, and the explicit formulae hold for exotic quantum permutation groups also.
\end{proof}

How far can these techniques be pushed?  In the case of $S_N^+$, the explicit formulae above together with the moments of the main character mean that it is possible to write down linear relations for the integrals of degree five monomials. Whether this gives a system of full rank is another matter.
\subsection*{Acknowledgement}
Some of this work goes back to discussions with Teo Banica, and free orbitals in $S_4^+$ are due to Teo. Thanks to David Roberson to pointing to Definition 4.1 in his pre-print \emph{Quantum symmetry vs nonlocal symmetry} with Simon Schmidt. The idea for the magic basis which gives the flat matrix model comes from this definition. The question of inner faithfulness was posed by Uwe Franz.

\end{document}